\documentclass[12pt,a4paper,oneside]{amsart}
\usepackage[foot]{amsaddr}
\usepackage{amsfonts,amssymb,amsthm,mathrsfs,eucal,amsmath,latexsym}
\usepackage{hyperref}
\usepackage[all]{xy}
\usepackage{diagrams}
\newarrow{Dash}{}{dash}{}{dash}{}
\newtheorem{theorem}{Theorem}

\newtheorem{prop}{Proposition}
\newtheorem{lemma}{Lemma}
\newcommand{\Hg}{\mathcal{H}}
\newcommand{\F}{\mathbb{F}}
\newcommand{\Z}{\mathbb{Z}}
\newcommand{\Q}{\mathbb{Q}}
\newcommand{\ord}{\mathrm{o}}
\newcommand{\C}{\mathscr{C}}
\newcommand{\G}{\mathscr{G}}
\newcommand{\A}{\mathscr{A}_{F}^{(p^\ell)}}
\begin{document}
\title{Extensions of degree $p^\ell$ of a $p$-adic field}
\author[M.R.Pati]{Maria Rosaria Pati}
\address{Dipartimento di Matematica, Universit\`{a} di Pisa\\
           Largo Pontecorvo 5\\
           56127 Pisa, Italia}
\email{pati@mail.dm.unipi.it}
\subjclass[2010]{11S05, 11S15}
\keywords{$p$-adic fields, Isomorphism classes of extensions, Galois theory, Ramification theory}
\begin{abstract}
Given a $p$-adic field $K$ and a prime number $\ell$, we count the total number of the isomorphism classes of $p^\ell$-extensions of $K$ having no intermediate fields. Moreover for each group that can appear as Galois group of the normal closure of such an extension, we count the number of isomorphism classes that contain extensions whose normal closure has Galois group isomorphic to the given group.
Finally we determine the ramification groups and the discriminant of the composite of all $p^\ell$-extensions of K with no intermediate fields.\\
The principal tool is a result, proved at the beginning of the paper, which states that there is a one-to-one correspondence between the isomorphism classes of extensions of degree $p^\ell$ of $K$ having no intermediate extensions and the irreducible $H$-submodules of dimension $\ell$ of $F^*/{F^*}^p$, where $F$ is the composite of certain fixed normal extensions of $K$ and $H$ is its Galois group over $K$.
\end{abstract}
\maketitle
\section{Introduction}
It is well known that a $p$-adic field $K$ has only a finite number of non-isomorphic algebraic extensions with given degree \cite{Kr}.
 We want to classify these extensions up to $K$-isomorphism in the totally and wildly ramified case, which is the only one slightly difficult to treat.\\
 In 1966 Krasner \cite{Kr} obtained an explicit formula for the total number of extensions of $K$ with given ramification index and inertial degree. Later, Serre \cite{Ser1} also computed the number of extensions using a different method in the proof of his famous \textquotedblleft mass formula''. In \cite{PR} Pauli and Roblot, with a more computational approach, gave another proof for the formula counting the number of extensions of a given degree and discriminant.\\
A new progress in the research was made in 2004, when Hou and Keating \cite{HK} considered the problem of determining the number of isomorphism classes of extensions of $K$ with given ramification index $e$ and inertial degree $f$; they found general formulas when $p^2\nmid e$ and, under some additional assumptions on $e$ and $f$, also when $p^2\parallel e$. This question was definitively closed by Monge in 2011 \cite{Mon}.\\
In 2007 Dvornicich and Del Corso \cite{DD}, looking at the isomorphism classes of extensions of $K$ of degree $p$, gave a new way to attach the problem of counting extensions: their idea is to \textquotedblleft shift" the $p$-extensions of $K$ in a more easy environment, where they can be identified by the action of a certain group on a suitable space. This new method has been used also by Dalawat \cite{Dal}, who extended it to the case of local fields of characteristic $p$.\\
Recently in \cite{DDM}, Del Corso, Dvornicich and Monge, taking inspiration from \cite{DD}, present a general and very useful way to study the extensions of degree $p^k$ of a $p$-adic field having no intermediate extensions. Denoting by $F$ the composite of all normal and tame extensions of $K$ whose Galois group is a subgroup of $\mathrm{GL}(k,\F_p)$ and $H=\mathrm{Gal}(F/K)$, they show that there exists a one-to-one correspondence between the possible extensions $\tilde L /K$ that appear as normal closure of extensions $L/K$ of degree $p^k$ having no intermediate fields and the irreducible $\F_p[H]$-submodules of $F^*/{F^*}^p$ of dimension $k$.\\
This key result allows to classify the extensions of $K$ of $p$-power degree only by studying the structure of the filtered $\F_p[H]$-module $F^*/{F^*}^p$; in other words, our problem is reduced to find the irreducible representation of dimension $k$ of a certain group $H$ acting on a suitable module $F^*/{F^*}^p$. This is practicable in the case in which $k$ is a prime number $\ell$, while in the general case this method is not indeed usable since the number of the possible representations increase with the number of the possible decomposition of $k$ as product of positive integers. The author intend to show the application of the same method to study the first non prime case $k=4$ in a forthcoming paper .\\
Specializing to the extensions of degree $p^\ell$, we obtain an improvement of the general theorem on the correspondence between isomorphism classes and irreducible modules, showing that in this case we can substitute $F$ with a smaller field, that is the composite of all normal extensions of $K$ of degree prime to $p$ whose Galois group is isomorphic to a subgroup of $\F_{p^\ell}^*$ not contained in $\F_p^*$ or to a non abelian subgroup of $\F_{p^\ell}^*\rtimes_\theta \mathrm{Gal}(\F_{p^\ell}/\F_p)$ where $\theta (\phi _p)=:{\phi _p}_{|\F_{p^\ell}^*}\in \mathrm{Aut}(\F_{p^\ell}^*)$ ($\phi _p$ the Frobenius automorphism).\\
Using this result, we classify the $p^\ell$-extensions of $K$ having no intermediate fields with respect to the Galois group of the normal closure, and count their total number up to $K$-isomorphism.\\
Finally, as a further application of the correspondence theorem, we determine the ramification groups and the discriminant of the composite of all extensions of degree $p^\ell$ of $K$ having no intermediate fields.

\section{Notation}
Let $p$ and $\ell$ prime numbers. For a $p$-adic field $K$ we denote by $e_K$ and $f_K$ the ramification index and the inertial degree of $K/\Q_p$ respectively, and by $n_K=e_K f_K=[K:\Q_p]$ the absolute degree of $K$.
We also denote by $\pi_K$ a uniformizer of $K$, by $\kappa_K$ its residue field and put $q_K=|\kappa_K|$.\\
If $E/K$ is a finite extension then $e_{E/K}$ and $f_{E/K}$ are the ramification index and the inertial degree of the extension; if $E/K$ is Galois with $\mathrm{Gal}(E/K)=G$ then $G=G_{-1}\supseteq G_0\supseteq G_1\supseteq \dots\supseteq \{1\}$ is, as usual, the lower numbering ramification filtration. In particular, $G_0$ is the inertia group and its fixed field $E^{\mathrm{ur}}$ is the maximal unramified subextension of $E/K$, while $G_1$ is the unique $p$-Sylow subgroup of $G_0$ and its fixed field is the maximal tame ramified subextension of $E/K$.\\
Finally we shall denote by $\C_{p^\ell}$ the composite of all extensions of $K$ of degree $p^\ell$ having no intermediate field.

\section{The correspondence theorem}
Let $F=F(K)$ be the composite of all normal extensions of $K$ of degree prime to $p$ whose Galois group is isomorphic to a subgroup of $\F_{p^\ell}^*$ not contained in $\F_p^*$ or to a non abelian subgroup of $\F_{p^\ell}^*\rtimes_\theta \mathrm{Gal}(\F_{p^\ell}/\mathbb{F}_p)$ where $\theta (\phi _p)=\phi _p|_{\F_{p^\ell}^*}\in \mathrm{Aut}(\F_{p^\ell}^*)$ ($\phi_p$ the Frobenius automorphism), and let $H=\mathrm{Gal}(F/K)$. Then
\begin{theorem}\label{mainimpr}
There exists a one to one correspondence between the isomorphism classes of extensions of degree $p^\ell$ having no intermediate extensions and the irreducible $H$-submodules of dimension $\ell$ of $F^*/{F^*}^p$.
\end{theorem}
\begin{proof}
We first show that to an extension $L/K$ of degree $p^\ell$ having no intermediate extensions one can associate an irreducible $H$-submodule of dimension $\ell$ of $F^*/{F^*}^p$. In order to obtain this, we will prove that $LF/F$ is an elementary abelian extension of degree $p^\ell$. It is easy to see that $[LF:F]=p^\ell$ since $L$ and $F$ are linearly disjoint over $K$. Observe that the extension $L/K$ is totally ramified, since otherwise it would have a proper subextension given by the maximal unramified subextension. Let $\tilde L$ be the normal closure of $L/K$ and $G=\mathrm{Gal}(\tilde L /K)$. As usual, $G$ has the lower numbering ramification filtration $G=G_{-1}\supseteq G_0\supseteq G_1\supseteq \cdots\supseteq \{1\}$, where for every $i$ the subgroup $G_i$ is normal in $G$ and for every $i\geq 1$ the quotient $G_i/G_{i+1}$ is an elementary abelian $p$-group.\\
Let $\tilde H \subseteq G$ be the subgroup fixing $L$. Since $\tilde L$ is the normal closure of $L/K$, no subgroup of $\tilde H$ is normal in $G$; it follows that the intersection of all its conjugates (which is normal if not trivial) is trivial. Moreover since $L/K$ has no intermediate extensions, $\tilde H$ is a maximal subgroup of $G$ and hence there is a unique $t$ such that $G_{t+1}\subseteq \tilde H$ and $G_t\tilde H=G$; observe that we must have $t\geq 1$ since $L/K$ is a totally and wildly ramified extension.\\
Now, $G_t$ is clearly a $\tilde H$-module (the action given by conjugation) and, since the centralizer $C_{\tilde H}(G_t)$ of $G_t$ in $\tilde H$ is trivial (being contained in the intersection of all conjugates of $\tilde H$), it is a faithful $\tilde H$-module. Moreover, $G_{t+1}=\{1\}$ since $\tilde H$ has no subgroup normal in $G$, therefore $G_t$ is an elementary abelian $p$-group; while from $C_{\tilde H}(G_t)=\{1\}$ we have $G_t\cap \tilde H=\{1\}$. It follows that $G\simeq G_t\rtimes\tilde H$ and $|G_t|=p^\ell$. This implies that $G_t$ is also irreducible as $\tilde H$-module since otherwise there would exists a proper $\tilde H$-submodule $A$ of $G_t$ and then a proper subgroup $A\rtimes\tilde H$ of $G$ containing $\tilde H$, which is a contradiction to the maximality of $\tilde H$.\\
Let $L_1$ be the subfield of $\tilde L$ fixed by $G_t$, so $\mathrm{Gal}(L_1/K)\simeq \tilde H$. We will show that the order of $\tilde H$ is prime to $p$. Since $G_t$ is a faithful $\tilde H$-module, $\tilde H$ can be embedded in $\mathrm{Aut}(G_t)\simeq \mathrm{GL}(\ell,\mathbb{F}_p)$. Let $\tilde H_1$ be the ramification group of $\tilde H$, then either $\tilde H_1=\{1\}$ or $\tilde H_1$ is the unique $p$-Sylow of $\tilde H_0$, and it is normal in $\tilde H$. But if $\tilde H$ has a non trivial normal $p$-subgroup, then $G_t$ would have a proper $\tilde H$-submodule, contradicting its irreducibility. So necessarily $p\nmid\ord(\tilde H)$.\\
Now we prove that $L_1$ is contained in $F$.\\
Since $p\nmid\ord(\tilde H)$, $\tilde H=\mathrm{Gal}(L_1/K)$ is the Galois group of a tame extension, therefore it is of the form
$$\langle v,\tau\mid v\tau v^{-1}=\tau^q, \tau^e=1, v^f=\tau^r\rangle$$ 
where $e=e_{L_1/K}$, $f=f_{L_1/K}$, $q=p^{f_K}$ and $r$ is the smallest positive integer such that $v^f=\tau^r$. Such an extension is not necessarily split, but it is always contained in a split one, say $L_2/K$, i.e. in a tame extension whose Galois group is a semidirect product of the inertia subgroup and its complement. $L_2/L_1$ is an unramified extension of suitable degree. Let $\tilde {\tilde H}=\mathrm{Gal}(L_2/K)$ and $\tilde{\tilde L}=\tilde L L_2$.\
Since $\tilde{\tilde H}$ is the Galois group of a split tame extension, we have
$$\tilde{\tilde H}=\langle\tilde\tau\rangle\rtimes\langle \tilde v\rangle$$
with $\ord(\tilde\tau)=e_{L_2/K}=e_{L_1/K}$, $\ord(\tilde v)=f_{L_2/K}$ and $\tilde v\tilde\tau{\tilde v}^{-1}={\tilde\tau}^q$.
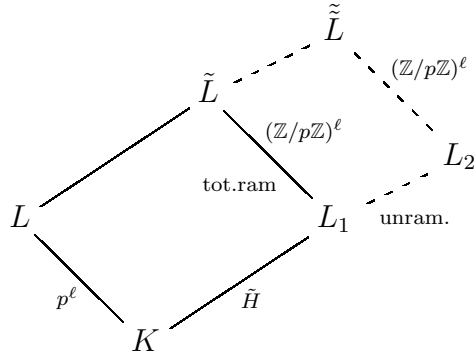
\begin{diagram}[size=2em,nohug,labelstyle=\scriptstyle ]
  &                       &     &            &             &    \tilde{\tilde L}   &     & \\
  &                       &     &  \tilde L  &          \ldDash(2,1)   &      &   \rdDash^{(\Z/p\Z)^\ell}  &  \\
  &            &  \ldLine(3,2)  &            &  \rdLine^{(\Z/p\Z)^\ell}_{\mathrm{tot.ram}} &                     &       & L_2  \\
L &                       &     &            &                          &       L_1          &                        \ldDash(2,1)_{\mathrm{unram.}}  &   \\
  & \rdLine(2,2)_{p^\ell} &     &            &   \ldLine(3,2)_{\tilde H} &                   &                          &   \\
  &                       &  K  &            &                           &                   &                          &\\
\end{diagram}
It easy to see that $\mathrm{Gal}(\tilde{\tilde L}/L_2)\simeq G_t\simeq (\Z/p\Z)^\ell$, $\mathrm{Gal}(\tilde{\tilde L}/L)\simeq \tilde{\tilde H}$ and thus $\mathrm{Gal}(\tilde{\tilde L}/K)\simeq G_t\rtimes_\rho\tilde{\tilde H}$. This means that there exists a representation $\rho $ of $\tilde{\tilde H}$ of dimension $\ell$ over $\F_p$ which is irreducible, since otherwise $G_t$ would have a proper $\tilde H$-submodule, but not necessarily faithful. Nevertheless, since the action of $\tilde H$ on $G_t$ is faithful and being $L_2/L_1$ unramified, $\rho$ is still faithful in $\langle\tilde\tau\rangle$.
\begin{lemma}\label{repr}
Every irreducible representation $V$ over $\F_p$ of $\langle\nu \rangle\rtimes_\mu \langle\eta \rangle$, where $\mu(\eta )$ is the elevation to $p^f$ and which is faithful on $\langle\nu \rangle$, is the sum of the conjugates of an irreducible representation over $\overline\F_p$, which is induced from a $1$-dimensional representation of $\langle\nu \rangle\rtimes\langle\eta_c\rangle$, where $\langle\eta_c\rangle=C_{\langle\eta\rangle}(\langle\nu\rangle)$ and $\nu $ and $\eta _c$ act as multiplication by $\alpha$ and $\beta$ respectively.\\
Moreover, the following equation holds 
\begin{equation}\label{dim1}
\mathrm{dim}_{\F_p}V=\mathrm{lcm}\left(\frac{rw}{(r,f)},r\right)
\end{equation}
where $r=[\F_p(\alpha):\F_p]$ and $w=[\F_p(\beta):\F_p]$.
\end{lemma}
\begin{proof}
See for example \cite{DDM}.
\end{proof}
By Lemma \ref{repr}, there exist $\alpha,\beta\in\overline\F_p^*$ such that $\rho$ is the sum of the conjugates of the representation obtained by induction from the $1$-dimensional representation on which $\tilde\tau$ and $\tilde v_c$ (with $\langle \tilde v_c\rangle=C_{\langle\tilde v\rangle}(\langle\tilde \tau\rangle)$ act as multiplication by $\alpha$ and $\beta$ respectively. Moreover, by equation \eqref{dim1}, $\alpha$ and $\beta$ are such that
$$\ell=\mathrm{lcm}\left (\frac{rw}{(r,f_K)},r\right).$$
Therefore we must have $r=1$ or $r=\ell$ (and also $w=1$ or $w=\ell$). This means that our representations over $\F_p$ are either the sum of the $\ell$ conjugates of a $1$-dimensional representation defined over $\F_{p^\ell}$, or it is the only conjugate of an induced representation from a $1$-dimensional one.\\
Since $\tilde H$ is a quotient of $\tilde{\tilde H}/\mathrm{ker}\rho$, we are interested in the image of $\tilde{\tilde H}$ under $\rho$.\\
If $r=1$ then $\alpha\in\F_p^*$, i.e. $e=\ord(\tilde\tau)=\ord(\alpha)\mid p-1$, therefore ${\tilde\tau}^q=\tilde\tau$ so that $\tilde{\tilde H}$ is abelian. It follows that the image of $\tilde{\tilde H}$ under $\rho$ is cyclic, isomorphic to the subgroup of $\F_{p^\ell}^*$ generated by $\alpha$ and $\beta$ (it is not contained in $\F_p$ since $w$ must be equal to $\ell$). Being one of its quotients, this is true also for $\tilde H$. Therefore in this case $L_1\subseteq F$.\\
If $r=\ell$, then we have to distinguish two cases: $\ell\mid f_K$ and $\ell\nmid f_K$. In the first case, we have again ${\tilde\tau}^q=\tilde\tau$ and hence $\tilde{\tilde H}$ abelian; since $r=\ell$, by the same argument of the previous case, $\tilde H$ is isomorphic to a subgroup of $\F_{p^\ell}^*$ not contained in $\F_p^*$. If $\ell\nmid f_K$ then $\tilde v$ acts as elevation to $p$ and $\tilde {\tilde H}$ is not abelian. But $\langle\tilde\tau\rangle\rtimes\langle{\tilde v}^\ell\rangle$ is abelian and its image under $\rho$ is cyclic, isomorphic to the subgroup $C$ of $\F_{p^\ell}^*$ generated by $\alpha$ and $\beta$ (so not contained in $\F_p$). The group $\tilde{\tilde H}$ has a subgroup isomorphic to the quotient $\frac{\tilde{\tilde H}}{\langle\tilde\tau\rangle\rtimes\langle{\tilde v}^\ell\rangle}\simeq \frac{\langle\tilde v\rangle}{\langle {\tilde v}^\ell\rangle}\simeq \Z/\ell\Z$, which acts on $\langle\tilde\tau\rangle\rtimes\langle{\tilde v}^\ell\rangle$ as elevation to $p$ so that $$\tilde{\tilde H}\simeq (\langle\tilde\tau\rangle\rtimes\langle{\tilde v}^\ell\rangle)\rtimes \frac{\tilde{\tilde H}}{\langle\tilde\tau\rangle\rtimes\langle{\tilde v}^\ell\rangle}\simeq C\rtimes\mathrm{Gal}(\F_{p^\ell}/\F_p)$$
with $C<\F_{p^\ell}^*$.
Recalling that $\tilde H$ is a quotient of $\tilde{\tilde H}/\mathrm{ker}\rho$, we find that $L_1$ is contained in $F$.\\
It follows that $\tilde L=LL_1\subseteq LF$. $\tilde L/L_1$ is totally and wildly ramified since $L/K$ is, therefore $\tilde L\cap F=L_1$ and $\mathrm{Gal}(LF/F)\simeq G_t$. This means that $LF/F$ is an elementary abelian extension of degree $p^\ell$. Thus, by Kummer theory, $LF=F(\sqrt[p]{\Xi })$ for some subgroup $\Xi $ of $F^*/{F^*}^p$ of dimension $\ell$ as $\F_p$-vector space. Moreover, $LF/K$ is Galois because $LF=\tilde{L}F$ and $\tilde{L}/K$ and $F/K$ are Galois, so $\Xi $ is a $H$-module (the action of $H$ being that induced by the action on $F^*$) and it is irreducible because otherwise $G_t$ would have a proper $\tilde H$-submodule.\\
Note that the conjugates of $L$ over $K$ are exactly the extensions which lead to the $H$-module $\Xi$ with this construction. Thus we can define a map $\Psi $ from the set of the isomorphism classes of $p^\ell$-extensions of $K$ having no intermediate fields to the set of the irreducible $H$-submodules of $F^*/{F^*}^p$ of dimension $\ell$ over $\F_p$.\\
\newline
Conversely, we show that each irreducible $H$-submodule of dimension $\ell$ of $F^*/{F^*}^p$ corresponds to the isomorphism class of an extension of degree $p^\ell$ of $K$ having no subextensions. Note that $p\nmid \ord(H)$ since $F$ is the composite of normal extensions of degree prime to $p$.\\
Let $\Xi $ be an irreducible $H$-submodule of $F^*/{F^*}^p$ which has dimension $\ell$ as vector space over $\F_p$. Put $M=F(\sqrt[p]{\Xi })$, then $M/K$ is a Galois extension and $S=\mathrm{Gal}(M/F)$ is a $H$-module which is irreducible because $\Xi $ is. Let $\mathfrak{G}$ be the Galois group of $M/K$. Since $S$ is a normal subgroup of $\mathfrak{G}$ of degree $p^\ell$ and $\mathfrak{G}/S\simeq H$ has order prime to $p$, by Schur-Zassenhaus theorem we have $\mathfrak{G}\simeq S\rtimes H$. Let $L$ be the fixed field of $H$; the fields fixed by the conjugates of $H$ in $\mathfrak{G}$ form the isomorphism class of the extension $L/K$. Each of these extensions has degree $p^\ell$ and has no intermediate extensions. In fact, let $T$ be such that $K\subseteq T\subseteq L$; $T$ is the field fixed by a subgroup $C$ of $\mathfrak{G}$, and since $L\supseteq T$ we have $C\supseteq H$. Therefore $C\simeq S_0\rtimes H$ with $S_0<S$, but $S$ is irreducible as $H$-module so $S_0=\{1\}$ or $S_0=S$ i.e. $T=L$ or $T=K$.\\
It follows that we can define a map $\Phi $ from the set of the irreducible $H$-submodules of dimension $\ell$ of $F^*/{F^*}^p$ to the set of the isomorphism classes of $p^\ell$-extensions of $K$ having no intermediate fields.\\
Finally it is easily seen that $\Psi$ and $\Phi$ are inverse to each other.
\end{proof}

\section{Counting the isomorphism classes}\label{counting}
Theorem \ref{mainimpr} says that to count the number of isomorphism classes of extensions of $K$ of degree $p^\ell$ it suffices to count the number of irreducible representations of $H$ of dimension $\ell$ in the $\F_p$-vector space $F^*/{F^*}^p$. To do this we need some information about $H$.\\
Note that $H$ is the Galois group of a finite tame extension therefore, as observed in the proof of the Theorem \ref{mainimpr}, it has the following form 
$$\langle v,\tau\mid v\tau v^{-1}=\tau^q, \tau^e=1, v^f=\tau^r\rangle$$
where $e:=e_{F/K}$, $f:=f_{F/K}$, $q:=p^{f_K}$ and $r$ is the smallest positive integer such that $v^f=\tau^r$, and the extension $F/K$ is always contained in a split one, i.e. in a tame extension whose Galois group is a semidirect product of the inertia subgroup and its complement. Since this tame split extension is of degree prime to $p$ and unramified over $F$, the Theorem \ref{mainimpr} still holds if we put it in place of $F$, as one can easily see by the proof of the same Theorem. With abuse of notation we continue to denote by $F$ this split tame extension of $K$ and by $H$ its Galois group over $K$.\\
So we can suppose that $H=H_0\rtimes U$ where $H_0=\langle \tau\rangle$ and $U=\langle v\rangle$ are cyclic of order $e$ and $f$ respectively, $(e,p)=1$, $e\mid q^f-1$ and $v$ acts on $H_0$ via the map $x\mapsto x^q$. This group satisfies the hypothesis of Lemma \ref{repr}, which describes its representations over $\F_p$.\\
For convenience, we now describe in more details the irreducible representations of $H$ over $\F_p$.\\
The irreducible representations of such a group over $\overline \F_p$ are easy to describe; using this description we determine that contained in the module obtained from $F^*/{F^*}^p$ by extension of scalar, and from these we recover the irreducible representation of $H$ over $\F_p$ contained in $F^*/{F^*}^p$.\\
First of all, we can observe that to study the irreducible representation $\rho$ of $H_0\rtimes U$ one can study the irreducible representation $\bar \rho$ of $\bar H_0\rtimes U$, where $\bar H_0=H_0/\mathrm{ker}(\rho_{| H_0})$.\\
As one can easily see for example in \cite{DDM}, all irreducible representations $\rho $ of $H_0\rtimes U$ in a $\overline\F_p$-vector space are induced from $1$-dimensional representations of the abelian group $H_0\rtimes \tilde U$, where $\tilde U=\langle \tilde v\rangle$ is the centralizer of $\bar H_0$ in $U$. In particular, if $\rho\colon H_0\rtimes U\longrightarrow \mathrm{GL}(W)$ is an irreducible representation, then $\rho =\mathrm{Ind}_{H_0\times \tilde U}^{H_0\rtimes U}(W_\chi )$ where $W_\chi$ is a $1$-dimensional $\overline\F_p$-subspace of $W$ on which $H_0\times \tilde U$ acts by the character $\chi$. Moreover, if $\alpha=\chi(\tau)$ and $\beta=\chi(\tilde v)$ then there is a basis of $W$ with respect to $\tau$ and $v$ act on $W$ via the matrices
\begin{equation*}
T_\alpha =
\begin{pmatrix}
\alpha\\
& \alpha ^q\\
& & \alpha^{q^2}\\
& & & \ddots\\
& & & & \alpha^{q^{s-1}}
\end{pmatrix}
, \quad
V_\beta =
\begin{pmatrix}
& & & & \beta\\
1\\
& 1\\
& & \ddots\\
& & & 1
\end{pmatrix}
\end{equation*}
where $s$ is the index of $\tilde U$ in $U$.\\
It easy to see that conjugate characters lead to the same representation on $W$.\\
Now, the second step is to pass from the irreducible representations over $\overline\F_p$ to the irreducible representations over $\F_p$. These last are sums of the conjugates of irreducible representations over $\overline\F_p$. If $\varphi $ is an irreducible representation of $H_0 \rtimes U$ over an $\F_p$-vector space $V$ then it is the sum of the conjugates of an irreducible representation over $\overline\F_p$ which, as written above, is induced from a $1$-dimensional one, and one has
\begin{equation}\label{dim}
\mathrm{dim}_{\F_p} V=\mathrm{lcm}\left(\frac{rw}{(r,f_K)},r\right)
\end{equation}
where $r=[\F_p(\alpha):\F_p]$ and $w=[\F_p(\beta):\F_p]$.\\
Finally it remains to identify the irreducible representations of $H_0\rtimes U$ over the $\F_p$-vector space $F^*/{F^*}^p$. For what we have said above, we first identify those over the $\overline \F_p$-vector space $F^*/{F^*}^p\otimes_{\F_p}\overline\F_p$ obtained from $F^*/{F^*}^p$ by extension of scalar and then recover from these the irreducible representations over $F^*/{F^*}^p$.

\subsection{The structure of $F^*/{F^*}^p$ as $\F_p[H]$-module}\label{reprirr}~\\

This section is due to the work of Del Corso, Dvornicich and Monge in \cite{DDM}, so for more details one can see their paper.\\
We describe the irreducible representations of $H$ of any dimension $k\geq 2$ contained in $F^*/{F^*}^p$, while in the next section we specialize to $k=\ell$ prime.\\
Recall that we are interested in representations of $H=H_0\rtimes U=\langle\tau\rangle\rtimes\langle v\rangle$ in a subspace of dimension $k$ of $F^*/{F^*}^p$, that is to say in the $\mathbb{F}_p [H]$-submodules of $F^*/{F^*}^p$ of dimension $k$.\\
If $\pi$ is a uniformizer of $F$ and $U_1$ denotes the group of principal units of $F$, it is well-known that $$F^*\simeq \langle \pi\rangle \times \kappa_F^*\times U_1$$ as $H$- modules, and that $U_1$ has a filtration $\{U_i\}_{i\geq 1}$, where $U_i =\{u\in F^* | u\equiv 1(\mathrm{mod}\: \pi ^i)\}$. It follows that $$F^*/{F^*}^p\simeq \langle \pi\rangle /{\langle \pi\rangle}^p\times U_1 / {U_1^p}$$ as $\mathbb{F}_p [H]$-modules. The filtration $\{U_i\}_{i\geq 1}$ of $U_1$ induces the filtration $\{U_i {U_1^p} /{U_1^p}\}_{i\geq 1}$ of $U_1 / {U_1^p}$ and, as $U_{i+1}$ is complemented in $U_i$ as $H$-module, also $U_{i+1} {U_1^p} /{U_1^p}$ is complemented in $U_i {U_1^p} /{U_1^p}$ as $\mathbb{F}_p [H]$-module. So one has $$F^*/{F^*}^p\simeq \langle \pi\rangle /{\langle \pi\rangle}^p\oplus \bigoplus_{i=1}^\infty U_i {U_1^p} / U_{i+1} {U_1^p} .$$
The nonzero terms in the right-hand side are those with $i=p e_F /(p-1)$ and, $0<i< p e_F /(p-1)$ with $(i,p)=1$ (see for example \cite{FV}). Then the above relation can be written as $$F^*/{F^*}^p\simeq \langle \pi\rangle /{\langle \pi\rangle}^p\oplus \bigoplus_{i\in [[0,I_F]]} U_i U_1^p / U_{i+1} U_1^p \oplus U_{I_F}U_1^p /U_1^p.$$ where $I_F=p e_F /(p-1)$ and $[[0,I_F]]$ is the set of integers prime with $p$ in the interval $]0,I_F[$.\\
We want to describe the representations of $H$ contained in $F^*/{F^*}^p$. To do this, observe that the action of $H$ on $\langle \pi\rangle /{\langle \pi\rangle}^p\simeq \mathbb{F}_p$ is clearly trivial and that $U_{I_F}U_1^p /U_1^p$ corresponds via Kummer theory to the Galois unramified extension of degree $p$, so the action of $H$ on this submodule of dimension $1$ of $F^*/{F^*}^p$ is given by the cyclotomic character $\omega $. Since we are interesting in the irreducible sub-representations of dimension $k\geq 2$, we can reduce to consider those contained in $\bigoplus_{i\in [[0,I_F]]} U_i U_1^p / U_{i+1} U_1^p$.\\
We need to study the structure of $U_i U_1^p / U_{i+1} U_1^p$ as $\mathbb{F}_p [H]$-module. Since $F/K$ is a tamely ramified extension, we can choose as a uniformizer $\pi$ of $F$ an $e$-th root of a uniformizer of $K$, where clearly $e=e_{F/K}$. Then for $i\geq 1$, each element of $U_i /U_{i+1}$ can be written as $1+\epsilon \pi^i$ with $\epsilon \in U_0$ ($U_0=\mathcal{O}_F^*$ is the multiplicative subgroup of the ring of the integers of $F$). The action of $H$ on it is given by $$\tau(1+\epsilon \pi^i)=1+\zeta ^i\epsilon \pi^i +\dots,\qquad v(1+\epsilon \pi^i)=1+\epsilon ^q\pi^i +\dots,$$ where $\zeta =\tau(\pi)/\pi $ is a primitive $e$-th root of $1$. As usual one can identify $U_i/U_{i+1}$ with $\kappa_F$ via the map $$1+\epsilon \pi^i\mapsto \bar{\epsilon }$$ and this induces on $\kappa_F$ the following action of $H$ $$\tau(\bar{\epsilon})=\bar{\zeta}^i\bar{\epsilon},\qquad v(\bar{\epsilon})=\bar{\epsilon}^q.$$
Denote by $M_i$ the $\mathbb{F}_p [H]$-module formed by the $\mathbb{F}_p$-module $\kappa_F$ with the above action of $H$. It is clear that $$U_i U_1^p/{U_{i+1}U_1^p}\simeq M_i$$ as $\mathbb{F}_p [H]$-modules. So we search the irreducible $\mathbb{F}_p [H]$-submodules of $\bigoplus_{i\in [[0,I_F]]} M_i$ or, equivalently, the irreducible representations of $H$ over $\mathbb{F}_p$ contained in $\bigoplus_{i\in [[0,I]]} M_i$ (viewed as $\mathbb{F}_p$-vector space). To do this we first extend the $\mathbb{F}_p$-representation $M_i$ of $H$ to an $\overline{\mathbb{F}}_p$-representation, identify its $\overline{\mathbb{F}}_p$ sub-representations and from each of these find the $\mathbb{F}_p$ sub-representations via the sum of its conjugates.\\
Let $\overline{M}_i=M_i\otimes _{\mathbb{F}_p} \overline{\mathbb{F}}_p$ be the extension of $M_i$ to the algebraic closure of $\mathbb{F}_p$. 
It can be shown that
\begin{equation}\label{barMi}
\overline{M}_i\simeq \bigoplus_{\substack{\beta\in\overline{\mathbb{F}}_p\\ \beta^{f/s}=1}}{J_{(\alpha,\beta)}}^{f_K}
\end{equation}
where $f=f_{F/K}$, $s=[U:\tilde U]$ and $J_{(\alpha,\beta)}=\mathrm{Ind}_{H_0\times \tilde{U}}^H(V_{(\alpha,\beta)})$ is the irreducible $s$-dimensional representation over $\overline\F_p$ induced from the $1$-dimensional representation on which $\tau$ and $\tilde v$ act via multiplication by $\alpha$ and $\beta$ respectively.\\
Let $Y=\bigoplus_{i\in [[0,I_F]]} \overline{M}_i$, then $J_{(\alpha,\beta)}$ appears $n_K=[K:\Q_p]$ times in $Y$. In fact, the $i$'s such that $\zeta^i=\alpha$ have equal remainder modulo $e$; being $(e,p)=1$ those in $]0,I_F[$ and prime with $p$ are exactly $e_K$. Multiplying by the exponent of $J_{(\alpha,\beta)}$ that appears in \eqref{barMi} we have the claim.\\
Moreover, as observed in the previous section, the $s$ conjugated pairs $(\alpha^{q^i},\beta)$ yield the same representation, therefore the multiplicity of $J_{(\alpha,\beta)}$ in $Y$ is $s n_K$.\\
It is easy to see that the representation $J_{(\alpha,\beta)}$ has $$d=\mathrm{lcm}(w,(r,f_K))$$ conjugates over $\mathbb{F}_p$, where $r=[\F_p(\alpha):\F_p]$ and $w=[\F_p(\beta):\F_p]$ as above, and so it is defined over $D=\mathbb{F}_{p^d}$. Moreover, a short reflection leads to observe that $s=[U:\tilde U]$ is equal to $\frac{r}{(r,f_K)}$ since $U$ acts on $H_0$ as elevation to $q$ and therefore $s$ is the smallest power of $q$ such that $e\mid q^s-1$.\\
Let $X$ be an irreducible sub-representation of $Y$ defined over $\mathbb{F}_p$ and containing a unique copy of $J_{(\alpha,\beta)}$. From each copy of $J_{(\alpha,\beta)}$ contained in $Y$ and defined over $D$ we obtain a representation isomorphic to $X$. Consequently, to count the sub-representations isomorphic to $X$ and defined over $\mathbb{F}_p$ is the same as to count the sub-representations that are isomorphic to $J_{(\alpha,\beta)}$ and defined over $D$. This can be made counting the sub-representations contained in $(J_{(\alpha,\beta)})^{s n_K}$ and working over $D$. Let's consider the immersions $$J_{(\alpha,\beta)}\longrightarrow (J_{(\alpha,\beta)})^{s n_K}$$ which are defined over $D$. Using Schur lemma, we find that the number of immersions is $|D|^{s n_K} -1$ and this number must be divided by $|D|-1$ when taking into account that two immersions have the same image if and only if they differ by multiplication by a constant. It follows that the number of representations defined over $\mathbb{F}_p$ and containing a representation isomorphic to $J_{(\alpha,\beta)}$ is $\frac{p^{ds n_K}-1}{p^d -1}$.\\

\subsection{The total number of isomorphism classes of extensions of degree $p^\ell$ of $K$}~\\

Let $V$ be an irreducible $\F_p[H]$-submodule of $F^*/{F^*}^p$ of dimension $\ell$. From what we have said above, viewing $V$ as irreducible representation of $H$ over $\F_p$, it is isomorphic to the sum of the conjugates over $\F_p$ of an irreducible representation of $H$ over $\overline\F_p$, which we have denoted by $J_{(\alpha,\beta)}$. If $r=[\F_p(\alpha):\F_p]$ and $w=[\F_p(\beta):\F_p]$, by equation \eqref{dim} we have
\begin{equation}\label{dim_ell}
\mathrm{dim}_{\mathbb{F}_p} V=\mathrm{lcm}\left(\frac{rw}{(r,f_K)},r\right)=\ell,
\end{equation}
therefore $r=1$ or $r=\ell$. It follows that $s=\frac{r}{(r,f_K)}=1$ or $s=\ell$. We must distinguish two cases: $\ell\mid f_K$ and $\ell\nmid f_K$.\\

If $\ell\mid f_K$ then $r$ and $w$ must be equal to $1$ or $\ell$ and at least one of them must be equal to $\ell$. It follows that we have exactly $(p^\ell-1)^2-(p-1)^2$ possible pairs $(\alpha,\beta)$; moreover, since $s=\frac{r}{(r,f_K)}=1$, for each of them there are $\frac{p^{\ell n_K}-1}{p^\ell-1}$ representations over $\F_p$ containing a representation isomorphic to $J_{(\alpha,\beta)}$. Finally, we have to take in account that the pairs $(\alpha,\beta),(\alpha^p,\beta^p),\dots,(\alpha^{p^{\ell-1}},\beta^{p^{\ell-1}})$ lead to the same representation over $\F_p$.\\
Collecting all the informations, we find that if $\ell\mid f_k$ there are exactly
$$\frac{1}{\ell}\frac{p^{\ell n_K}-1}{p^\ell-1}((p^\ell-1)^2-(p-1)^2)$$
isomorphism classes of extensions of degree $p^\ell$ of $K$ having no intermediate fields.\\

If $\ell\nmid f_K$ then equation \eqref{dim_ell} says that one of $r$ and $w$ must be $1$ and the other $\ell$.\\
For $r=1$ and $w=\ell$, we have again $s=\mathrm{dim}J_{(\alpha,\beta)}=1$ and $d=\ell$ therefore, for each of the $(p-1)(p^\ell-1)-(p-1)^2$ possible choices of $(\alpha,\beta)$, there are $\frac{p^{\ell n_K}-1}{p^\ell-1}$ $\F_p$-representations containing a representation isomorphic to $J_{(\alpha,\beta)}$, and for the same reason as above this number must be divided by $\ell$. So for $\alpha\in\F_p^*$ and $\beta\in\F_{p^\ell}^*$ we have $\frac{1}{\ell}\frac{p^{\ell n_K}-1}{p^\ell-1}(p-1)(p^\ell-p)$ irreducible representations of $H$ of dimension $\ell$ over the $\F_p$-vector space $F^*/{F^*}^p$.\\
To these we have to add that obtained from $r=\ell$ and $w=1$. In this case $s=\ell$ and $d=1$ i.e. $J_{(\alpha,\beta)}$ is already defined over $\F_p$, therefore for each of the $(p^\ell-1)(p-1)-(p-1)^2$ possible pairs $(\alpha,\beta)$ there are $\frac{p^{\ell n_K}-1}{p-1}$ representations isomorphic to $J_{(\alpha,\beta)}$. This value needs to be divided by $\ell$ because from the definition of $J_{(\alpha,\beta)}$ one has $J_{(\alpha,\beta)}=J_{(\alpha^p,\beta)}=\cdots =J_{(\alpha^{p^{\ell-1}},\beta)}$. It follows that for $\alpha\in \F_{p^\ell}^*$ and $\beta\in \F_p^*$ there are $\frac{1}{\ell}(p^{\ell n_K}-1)(p^\ell-p)$ irreducible representations of $H$ of dimension $\ell$ over $F^*/{F^*}^p$.\\
Adding the two contributions one finds
\begin{multline*}
\frac{1}{\ell}(p^{\ell n_K}-1)(p^\ell -p)+\frac{1}{\ell}\frac{p^{\ell n_K}-1}{p^\ell-1}[(p-1)(p^\ell -1)-(p-1)^2]= \\
 \frac{1}{\ell}\frac{p^{\ell n_K}-1}{p^\ell-1}[(p^\ell -1)^2-(p-1)^2]
\end{multline*}
isomorphism classes of extensions of degree $p^\ell$ of $K$ having no intermediate fields.\\
Note that we have obtained the same number of isomorphism classes of extensions as in the case $\ell\mid f_K$.\\

\subsection{The number of isomorphism classes of extensions whose normal closure has a prescribed Galois group}~\\

In the previous section we have count all the isomorphism classes of extensions of degree $p^\ell$ of $K$ having no intermediate fields. To each of them we can associate a group, that is the Galois group of the normal closure of the extensions over $K$. It turns that some isomorphism classes are associated to the same group, i.e. non isomorphic $p^\ell$-extensions of $K$ can have normal closure with the same Galois group.\\
In this section we identify all possible groups that can appear as the Galois group of the normal closure of a $p^\ell$-extension of $K$ having no intermediate fields, and for each of them we count the number of isomorphism classes of extensions which are associated to it.\\

First of all observe that if $L/K$ is a $p^\ell$-extension having no intermediate fields and $\tilde L$ is its normal closure then, by Theorem \ref{mainimpr}, $$\mathrm{Gal}(\tilde L/K)\simeq V\rtimes_{\bar\rho}\bar H$$ where $\bar\rho$ is the map induced on the quotient $\bar H=H/{\mathrm{ker}\rho}$ and the pair $(V,\rho)$ is the representation of dimension $\ell$ of $H$ in $F^*/{F^*}^p$, which correspond to the class of $L/K$ under the correspondence of Theorem \ref{mainimpr}. In other words, $\mathrm{Gal}(\tilde L/K)$ is the semidirect product of $V$ with the biggest quotient of $H$ acting faithfully on it.\\
Fixing a basis of the $\F_p$-vector space $V$, we can identify the image of $\rho$ with a subgroup of $\mathrm{GL}(\ell,\F_p)$, so that $\mathrm{Gal}(\tilde L/K)\simeq (\F_p^+)^\ell\rtimes \Hg_\rho$ where $\Hg_\rho$ represents the action of $\bar H$ on $V$ given by $\bar\rho$, expressed with respect to the fixed basis.\\
Now observe that for what we have said above, our representations are sums of the conjugates of $s$-dimensional representations $J_{(\alpha,\beta)}$, where $s=1$ or $s=\ell$ and $\alpha,\beta\in\F_{p^\ell}^*$: if $s=1$ then $J_{(\alpha,\beta)}$ is defined over $\F_{p^\ell}$ and $\bar H$ is abelian, so $\Hg_\rho=\Hg_{(\alpha,\beta)}$ is cyclic since it is isomorphic to the homomorphic image of a finite group in $\mathrm{GL}(1,\F_{p^\ell})=\F_{p^\ell}^*$; if $s=\ell$ then $J_{(\alpha,\beta)}$ is already defined over $\F_p$ but the group $\Hg_{(\alpha,\beta)}$ of the matrices which describes the action of $\bar H$ is not abelian.\\
The normal closures of two isomorphism classes have the same Galois group if and only if $$(\F_p^+)^\ell\rtimes \Hg_{(\alpha,\beta)}\simeq (\F_p^+)^\ell\rtimes \Hg_{(\alpha',\beta')}$$ and this happens if and only if the two subgroups $\Hg_{(\alpha,\beta)}$, $\Hg_{(\alpha',\beta')}$ of $\mathrm{GL}(\ell,\F_p)$ are conjugated over $\mathrm{GL}(\ell,\F_p)$. In fact, if $\sigma\colon (\F_p^+)^\ell\rtimes \Hg_{(\alpha,\beta)}\longrightarrow (\F_p^+)^\ell\rtimes \Hg_{(\alpha',\beta')}$ is an isomorphism then for every $A\in \Hg_{(\alpha,\beta)}$ the following diagram must be commutative
\begin{equation*}
\xymatrix{
\F_p^\ell \ar[d]_{\sigma|_{\F_p^\ell}} \ar[r]^A & \F_p^\ell \ar[d]^{\sigma|_{\F_p^\ell}}\\
\F_p^\ell \ar[r]_{\sigma(A)} & \F_p^\ell }
\end{equation*}
where $\sigma |_{\F_p^\ell}\in \mathrm{Aut}(\F_p^\ell)$ and thus can be expressed as an invertible matrix in $\mathrm{GL}(\ell,\F_p)$.\\
To classify the various case depending on the value of $s=r/(r,f_K)$, we must distinguish two different situations: $\ell\mid f_K$ and $\ell\nmid f_K$.\\ 
For what will follow it is convenient to introduce the function $\psi (a,b)$ that maps $(a,b)\in\mathbb{N}\times\mathbb{N}$ to the number of elements of order $a$ in the group $C_a\times C_b$. It can be expressed as 
\begin{equation*}
\psi (a,b)=a\cdot (a,b)\cdot\prod_{\substack {l \:\mathrm{prime}\\ l 
\mid (a,b)}}\left(1-\frac{1}{l^2}\right)\cdot\prod_{\substack{ l \:\mathrm{prime}\\ l\mid a\\ l \nmid (a,b)}}\left(1-\frac{1}{l}\right).
\end{equation*}
Moreover, for every $c$ dividing $p^\ell-1$ we define the function $\lambda(c,p)$ as
\begin{equation}\label{lambda}
\lambda(c,p)=
\begin{cases}
1   & \text{if $p\equiv 2,\dots,\ell-1\:(\mathrm{mod}\:\ell)$ or}\\ 
    & \text{$p\equiv 1\:(\mathrm{mod}\:\ell)$ and, $v_\ell(c)=0$ or $v_\ell(c)=v_\ell(p^\ell -1)$}\\
\frac{1}{\ell} & \text{if $p\equiv 1\:(\mathrm{mod}\:\ell)$ and $v_\ell(p-1)<v_\ell(c)<v_\ell(p^\ell -1)$}\\
\frac{1}{\ell+1} & \text{if $p\equiv 1\:(\mathrm{mod}\:\ell)$ and $v_\ell(c)\leq v_\ell(p-1)$}
\end{cases}
\end{equation}
\newline
CASE $\ell\mid f_K$.
\begin{theorem}
Let $K$ be a $p$-adic field, $f_K$ its inertial degree over $\Q_p$ and $n_K$ its absolute degree. Let $\ell$ be a prime number and suppose that $\ell\mid f_K$. Then the Galois group of the normal closure of a $p^\ell$-extension of $K$ having no intermediate fields is of type $\F_{p^\ell}^+\rtimes C$, where $C$ is a subgroup of $\F_{p^\ell}^*$ not contained in $\F_p^*$ with the natural action on $\F_{p^\ell}^+$.\\
Moreover, for every integer $c$ dividing $p^\ell-1$ but not $p-1$, if $C$ is the cyclic subgroup of $\F_{p^\ell}^*$ of order $c$, then there are
$$n(c)=\frac{1}{\ell}\psi(c,p^\ell-1)\frac{p^{\ell n_K}-1}{p^\ell-1}$$
classes of isomorphic $p^\ell$-extensions of $K$ having no intermediate fields whose normal closure has Galois group isomorphic to $\F_{p^\ell}^+\rtimes C$.
\end{theorem}
\begin{proof}
Since $\ell$ is a prime number, from (\ref{dim}) it follows that in this case for the dimension to be $\ell$ we need $r,w$ equal to $1$ or $\ell$ and at least one to be equal to $\ell$. Moreover for every $\alpha$ and $\beta$ satisfying this condition, we always have $\mathrm{dim}J_{(\alpha,\beta)}=s=\frac{r}{(r,f_K)}=1$. It follows that $\Hg_{(\alpha,\beta)}$ is generated by the diagonal matrices
\begin{equation*}
\begin{pmatrix}
\alpha\\
& \alpha ^p\\
& & \ddots\\
& & & \alpha^{p^{\ell-1}}
\end{pmatrix}
, \quad
\begin{pmatrix}
\beta\\
& \beta ^p\\
& & \ddots\\
& & & \beta^{p^{\ell-1}}
\end{pmatrix}
\end{equation*}
and therefore it is isomorphic to the cyclic subgroup of $\F_{p^\ell}^*$ generated by $\alpha$ and $\beta$. Therefore the only case in which $(\F_p^+)^\ell\rtimes \Hg_{(\alpha,\beta)}\simeq (\F_p^+)^\ell\rtimes \Hg_{(\alpha',\beta')}$ is when $\{\alpha,\beta\}$ and $\{\alpha',\beta'\}$ generate groups of the same order.\\
Note that if $\alpha$ and $\beta$ generate the cyclic subgroup $C$ of $\F_{p^\ell}^*$ then $(\F_p^+)^\ell\rtimes \Hg_{(\alpha,\beta)}\simeq\F_{p^\ell}^+\rtimes  C$, where $C$ acts in the natural way on $\F_{p^\ell}^+$.\\ 
Let $c\mid p^\ell -1$ but $c\nmid p-1$, then the number of pairs $(\alpha,\beta)$ in $\mathbb{F}_{p^\ell}^*\times\mathbb{F}_{p^\ell}^*$ having order $c$ is equal to $\psi (c,p^\ell -1)$. Since $s=1$ for each of them, the number of representations in $F^*/{F^*}^p$ defined over $\mathbb{F}_p$ and containing a representation isomorphic to $J_{(\alpha,\beta)}$ is $(p^{\ell n_K}-1)/(p^\ell -1)$. Observe that we need to count together $(\alpha,\beta)$, $(\alpha^p,\beta^p)$ $\dots$ $(\alpha^{p^{\ell-1}},\beta^{p^{\ell-1}})$ since they lead to the same representation over $\F_p$, so we must divide by $\ell$.\\
It follows that if $C$ is the cyclic group of order $c$ in $\mathbb{F}_{p^\ell}^*$, the number of classes of extensions whose normal closure has Galois group isomorphic to $\mathbb{F}_{p^\ell}^+\rtimes  C$ (where $C$ acts naturally on $\F_{p^\ell}^+$) is exactly $$\frac{1}{\ell}\psi(c,p^\ell -1)\frac{p^{\ell n_K}-1}{p^\ell -1}.$$
The groups of type $\F_{p^\ell}^+\rtimes C$ with $C<\F_{p^\ell}^*$ is the only groups that can appear as Galois group of the normal closure of a $p^\ell$-extension of $K$ having no intermediate fields when $\ell\mid f_K$.
\end{proof}

\noindent CASE $\ell\nmid f_K$. 
\begin{theorem}
Let $K$ be a $p$-adic field, $f_K$ its inertial degree over $\Q_p$ and $n_K$ its absolute degree. Let $\ell$ be a prime number and suppose that $\ell\nmid f_K$. Then the Galois group of the normal closure of a $p^\ell$-extension of $K$ having no intermediate fields is either of type $\F_{p^\ell}^+\rtimes C$, where $C$ is a subgroup of $\F_{p^\ell}^*$ not contained in $\F_p^*$ (with the natural action on $\F_{p^\ell}^+$), or of type $\left(\F_p\right)^\ell\rtimes \Hg$, where $\Hg\subseteq \mathrm{GL}(\ell,\F_p)$ is isomorphic to a non abelian subgroup of $\F_{p^\ell}^*\rtimes\mathrm{Gal}(\F_{p^\ell}/\F_p)$.\\
Moreover, 
\begin{itemize}
\item for every integer $c$ dividing $p^\ell-1$ but not $p-1$, if $C$ is the cyclic subgroup of $\F_{p^\ell}^*$ of order $c$, then there are
$$n(c)=\frac{1}{\ell}\psi(c,p-1)\frac{p^{\ell n_K}-1}{p^\ell-1}$$
classes of isomorphic $p^\ell$-extensions of $K$ having no intermediate fields whose normal closure has Galois group isomorphic to $\F_{p^\ell}^+\rtimes C$;
\item for every non abelian subgroup $\Hg\subseteq \mathrm{GL}(\ell,\F_p)$ isomorphic to a subgroup $Z$ of $\F_{p^\ell}^*\rtimes\mathrm{Gal}(\F_{p^\ell}/\F_p)$, if $C=Z\cap \F_{p^\ell}^*$ has order $c$ then there are
\begin{equation*}
n(\Hg)=
\begin{cases}
\lambda(c,p)\psi(c,p-1)\frac{1}{\ell}\frac{p^{\ell n_K}-1}{p-1} & \text{if $C\rightarrow Z$ splits}\\
\frac{1-\lambda(c,p)}{\ell-1}\psi(c,p-1)\frac{1}{\ell}\frac{p^{\ell n_K}-1}{p-1} & \text{if $C\rightarrow Z$ does not split}
\end{cases}
\end{equation*}
classes of isomorphic $p^\ell$-extensions of $K$ having no intermediate fields whose normal closure has Galois group isomorphic to $\left(\F_p\right)^\ell\rtimes \Hg$.
\end{itemize}
\end{theorem}
\begin{proof}
Since $\ell\nmid f_K$ from \eqref{dim}, the dimension of a representation of a quotient $\bar H$ over $\mathbb{F}_p$ is $\ell$ when one of $r,w$ is $1$ and the other $\ell$.\\
For $r=1$ and $w=\ell$, i.e. for the pairs $\alpha$, $\beta$ with $\alpha\in \F_p^*$ and $\beta\in \F_{p^\ell}^*$, it turns out again $\dim J_{(\alpha,\beta)}=s=1$ so, as above, the group $\Hg_{(\alpha,\beta)}$ acting in the representation is cyclic of order equal to the order of $(\alpha,\beta)$ in $(\mathbb{F}_p^*\times\mathbb{F}_{p^\ell}^*)\setminus (\mathbb{F}_p^*\times\mathbb{F}_p^*)$ and for each pair $(\alpha,\beta)$ there are $(p^{\ell n_K}-1)/(p^\ell -1)$ representations over $\F_p$ containing a unique copy of $J_{(\alpha,\beta)}$.\\
Let $c\mid p^\ell -1$ but $c\nmid p-1$, the possible pairs $(\alpha,\beta)$ in $\mathbb{F}_p^*\times\mathbb{F}_{p^\ell}^*$ of order $c$ are $\psi(c,p-1)$. Thus, similarly to above, if $C$ is the cyclic group of order $c$ in $\mathbb{F}_{p^\ell}^*$ then the number of classes of extensions whose normal closure has Galois group isomorphic to $\mathbb{F}_{p^\ell}^+\rtimes C$ is $$\frac{1}{\ell}\psi(c,p-1)\frac{p^{\ell n_K}-1}{p^\ell -1}.$$ 
For $r=\ell$ and $w=1$ we have $\dim J_{(\alpha,\beta)}=s=\frac{r}{(r,f_K)}=\ell$ and $J_{(\alpha,\beta)}$ is already defined over $\F_p$. It follows that the group $\Hg_{(\alpha,\beta)}$ coincides with the group of matrices which describe the action on $J_{(\alpha,\beta)}$.\\
Recalling that $q=p^{f_K}$ and $\ell\nmid f_K$, from the study of the representations of $H$ we have made at the beginning of section \ref{counting}, we find that the action on $J_{(\alpha,\beta)}$ is described by the matrices
\begin{equation*}
T_\alpha =
\begin{pmatrix}
\alpha\\
& \alpha ^p\\
& & \ddots\\
& & &\alpha^{p^{\ell-1}}
\end{pmatrix}
, \quad
V_\beta =
\begin{pmatrix}
& & &\beta\\
1\\
& \ddots\\
& & 1
\end{pmatrix}.
\end{equation*}
So $\Hg_{(\alpha,\beta)}=\langle T _\alpha,V_\beta\rangle$ is a non-abelian group. Moreover the number of representations contained in $F^*/{F^*}^p$ isomorphic to $J_{(\alpha,\beta)}$ is exactly $(p^{\ell n_K}-1)/(p-1)$. This number must be divided by $\ell$ when taking into account that the pairs $(\alpha,\beta)$, $(\alpha^p,\beta)$ $\dots$ $(\alpha^{p^{\ell-1}},\beta)$ in $\mathbb{F}_{p^\ell}^*\times\mathbb{F}_p^*$ give the same representation.\\
It remains to multiply by the number of pairs $(\alpha',\beta')$ such that $(\F_p^+)^\ell\rtimes\langle T_{\alpha'},V_{\beta'}\rangle\simeq (\F_p^+)^\ell\rtimes\langle T_\alpha,V_\beta\rangle$. For what we have remarked above, this means that we have to count the number of pairs $(\alpha',\beta')\in \F_{p^\ell}^*\times\F_p^*$ such that $\langle T_{\alpha'},V_{\beta'}\rangle$ and $\langle T_\alpha,V_\beta\rangle$ are conjugated in $\mathrm{GL}(\ell,\F_p)$ (note that this is equivalent to require that they are conjugated in $\mathrm{GL}(\ell,\F_{p^\ell})$).\\ 
It turns out that $\Hg_{(\alpha,\beta)}$ is conjugated to $\Hg_{(\alpha',\beta')}$ if only if there exists a matrix $M\in\mathrm{GL}(\ell,\F_p)$ such that $M^{-1}\Hg_{(\alpha,\beta)}M=\Hg_{(\alpha',\beta')}$ and diagonal matrices are sent in diagonal matrices.\\
In fact, the subgroup of the diagonal matrices of $\Hg_{(\alpha,\beta)}$ is generated by $T_\alpha$ and $V_\beta^\ell$, it is a maximal cyclic subgroup $\langle T_\gamma\rangle$ of $\Hg_{(\alpha,\beta)}$ of order equal to $\ord(\gamma)=\mathrm{lcm}(\ord(\alpha),\ord(\beta))$ and index $\ell$.
Suppose that $\Hg_{(\alpha,\beta)}$ and $\Hg_{(\alpha',\beta')}$ are conjugated, then the maximal cyclic subgroups of the diagonal matrices have the same order, and in particular they are equal since they are isomorphic to the same subgroup of $\F_{p^\ell}^*$ via the same map. Therefore $\Hg_{(\alpha,\beta)}=\langle T_\gamma, V_\beta\rangle$ and $\Hg_{(\alpha',\beta')}=\langle T_\gamma, V_{\beta'}\rangle$.\\
Now, $\Hg_{(\alpha,\beta)}$ and $\Hg_{(\alpha',\beta')}$ have either only one maximal cyclic subgroup, that is the subgroup of the diagonal matrices, or $\ell+1$ maximal cyclic subgroups (one of which is that of the diagonal matrices). In the first case, there is nothing to prove since clearly all the conjugations send diagonal matrices in diagonal matrices. For the second case, observe that by the proof of Theorem \ref{mainimpr}, $\Hg_{(\alpha,\beta)}$ is isomorphic to $\langle(\gamma,\mathrm{id}),(\gamma,\phi _p)\rangle\subseteq \F_{p^\ell}^*\rtimes\mathrm{Gal}(\F_{p^\ell}/\F_p)$, i.e. it is isomorphic to a subgroup of the normalizer of a Cartan subgroup of $\mathrm{GL}(\ell,\F_p)$.\\
The key point is to observe that there always exists a basis under which $(\gamma,\mathrm{id})$ and $(\gamma,\phi_p)$ are represented by the generators of any two cyclic subgroups of index $\ell$ of $\Hg_{(\alpha,\beta)}$ and every pairs of these matrices generate the whole group $\Hg_{(\alpha,\beta)}$.
This is enough to prove that we can suppose not only $\langle T_\gamma\rangle$ is sent in $\langle T_\gamma\rangle$ but also that $\langle T_{\alpha^\ell}, V_\beta\rangle$ goes in $\langle T_{\alpha'^\ell},V_{\beta'}\rangle$.\\
Therefore there exists a diagonal matrix $D=T_\alpha^i V_\beta^{\ell j}\in\Hg_{(\alpha,\beta)}$ such that $M^{-1}(D V_\beta)M=V_{\beta'}$. This leads to $$\beta'=\gamma^{p^{\ell-1}+\cdots +p+1}\beta,$$ where recall that $\gamma\in\F_{p^\ell}^*$ has order $\mathrm{lcm}(\ord(\alpha),\ord(\beta))=\mathrm{lcm}(\ord(\alpha'),\ord(\beta'))$.\\ 
Consequently the Galois group is identified by the order of the cyclic group $C=\langle\gamma\rangle\subseteq \F_{p^\ell}^*$ and the class of $\beta$ in $(C\cap \F_p^*)/C^{p^{\ell-1}+\cdots +p+1}$.\\
Let $c$ be the order of $C$.
Since
\begin{equation*}
(p-1,p^{\ell -1}+\cdots +p+1)=
\begin{cases}
\ell & \text{if $p\equiv 1\:(\mathrm{mod}\:\ell)$}\\
1 & \text{if $p\equiv 2,\dots,\ell -1\:(\mathrm{mod}\:\ell)$ or $p=\ell$}
\end{cases}
\end{equation*}
we have
\begin{equation*}
\left | \frac{C\cap \mathbb{F}_p^*}{C^{p^{\ell -1}+\cdots +p+1}}\right | =
\begin{cases}
\ell & \text{if $p\equiv 1\:(\mathrm{mod}\:\ell)$ and $0<v_\ell(c)<v_\ell(p^\ell -1)$}\\
1 & \text{otherwise}.
\end{cases}
\end{equation*}
Clearly, if $|(C\cap \F_p^*)/C^{p^{\ell -1}+\cdots+p+1}|=1$ then there is only one class $\mathrm{mod}\ C^{p^{\ell-1}+\cdots+p+1}$, therefore every pairs $(\alpha,\beta)$ such that $\alpha$ and $\beta$ generate $C$ give (class of) extensions with isomorphic Galois groups.\\
We now consider the case $|(C\cap \F_p^*)/C^{p^{\ell -1}+\cdots+p+1}|=\ell$.\\
Suppose $\ell^k\parallel c$, then $\beta\in C^{p^{\ell-1}+\cdots+p+1}$ if and only if its order is not divisible by a power of $\ell$ greater than $\ell^k/(\ell^k,p^{\ell-1}+\cdots+p+1)$.\\
On the other hand, we have that $\beta\in C^{p^{\ell-1}+\cdots+p+1}$ if and only if the map $C\hookrightarrow \langle T_\alpha,V_\beta\rangle$ splits (with abuse of notation, we have identified $C$ we its image in $\langle T_\alpha,V_\beta\rangle$). In fact, this map splits if and only if $\beta$ has an $\ell$-root in $C$, i.e. there exists $t\in C$ such that $t^\ell=\beta$. This is equivalent to say that there exists $r\in C$ such that $r^{p^{\ell-1}+\cdots+p+1}=\beta$ since $C^{p^{\ell-1}+\cdots+p+1}=(C\cap \F_p^*)^{p^{\ell-1}+\cdots+p+1}=(C\cap \F_p^*)^\ell$.\\
We write $\F_{p^\ell}^*\times\F_p^*$ as direct sum of $l$-groups for each prime $l$, and chose the components of $(\alpha,\beta)$ in this sum among the elements of the correct order. For what we have said above, for $l\neq \ell$ the choice of the $l$-component of $(\alpha,\beta)$ has no effect on the condition $\beta\in C^{p^{\ell-1}+\cdots +p+1}$.\\
For $l=\ell$, the $\ell$-part of $\F_{p^\ell}^*\times\F_p^*$ is $C_{\ell^{m+z}}\times C_{\ell^z}$ where $\ell^m \parallel p^{\ell-1}+\cdots +p+1$ and $\ell^z \parallel p-1$. Recall that $\ell^k \parallel c$ for some $1\leq k<m+z$.
Since we are in the case $(p-1,p^{\ell -1}+\cdots +p+1)=\ell$, we have $z=1$ or $m=1$.\\
If $z=1$ then $C^{p^{\ell-1}+\cdots+p+1}$ is trivial, therefore $\beta\in C^{p^{\ell-1}+\cdots +p+1}$ when the $\ell$-component of $(\alpha,\beta)$ is of type $(x,1)$, and this happens for $\frac{\phi (\ell^k)}{\ell\phi (\ell^k)}=\frac{1}{\ell}$ of the possible $\ell$-components when $k>1$ and for $\frac{\ell-1}{\ell^2 -1}=\frac{1}{\ell +1}$ of the possible $\ell$-components when $k=1$.\\
If $m=1$ then $\beta\in C^{p^{\ell-1}+\cdots +p+1}$ if its order is not divisible by a power of $\ell$ greater than $\ell^{k-1}$, therefore the $\ell$-component of $(\alpha,\beta)$ must be of type $(x,y)$ with the order of $x$ greater than the order of $y$; this is true for $\frac{\phi (\ell^k)\ell^{k-1}}{2\phi (\ell^k)\ell^k-{\phi (\ell^k)}^2}=\frac{1}{\ell+1}$ of the elements with correct order.\\
Recollecting all the informations achieved, we have that if $\lambda$ is the function defined in \eqref{lambda}, then $\lambda(c,p)\psi(c,p-1)$ is the number of pairs $(\alpha,\beta)$ such that the group $C$ generated by $\alpha$ and $\beta$ has order $c$ and $\beta\in C^{p^{\ell-1}+\cdots +p+1}$, while $(1-\lambda(c,p))\psi(c,p-1)$ counts the same pairs but with $\beta\notin C^{p^{\ell-1}+\cdots +p+1}$.\\
Finally, to conclude the proof it remains to observe that the $\beta$'s not contained in $C^{p^{\ell-1}+\cdots +p+1}$ are equally distributed in each of the $\ell-1$ non trivial class $\mathrm{mod}\ C^{p^{\ell-1}+\cdots+p+1}$.
\end{proof}

\section{Ramification groups and discriminant of the composite of all $p^\ell$-extensions of $K$}
We know that there is a one-to-one correspondence between the isomorphism classes of  extensions of degree $p^\ell$ of $K$ having no intermediate extensions and the irreducible $H$-submodules of dimension $\ell$ of $F^*/{F^*}^p$, or equivalently, the abelian $p^\ell$-extensions of $F$ of exponent $p$, Galois over $K$, having no intermediate subextensions that are Galois over $K$.\\
Denote by $\C_{p^\ell}$ the composite of all extensions of degree $p^\ell$ of $K$ having no intermediate fields. Clearly $\C_{p^\ell}/K$ is Galois, we want to determine the ramification groups of its Galois group.\\
Let $\G=\mathrm{Gal}(\C_{p^\ell}/K)$, as usual for $i\geq -1$ we denote by $\G_i$ the $i$-th ramification group of $\G$.
\begin{lemma}\label{C=A}
One has $$\C_{p^\ell}=\A$$
where $\A$ is the composite of all abelian $p^\ell$ extensions of $F$ of exponent $p$ having no other subextensions that are Galois over $K$.
\end{lemma}
\begin{proof}
If $L/K$ is a $p^\ell$-extension having no intermediate fields then $LF$ is contained in $\A$ (see proof of Theorem \ref{mainimpr}), therefore $\C_{p^\ell}\subseteq \A$.\\
Conversely, let $E$ be an abelian $p^\ell$-extension of $F$ of exponent $p$, Galois over $K$ but having no subextensions that are Galois over $K$. As seen in the proof of Theorem \ref{mainimpr}, $\mathrm{Gal}(E/K)\simeq \Xi \rtimes H$ where $\Xi $ is the $\F_p[H]$-submodule of $F^*/{F^*}^p$ of dimension $\ell$ which corresponds to $E/K$, $E^{H}$ is contained in $\C_{p^\ell}$ and $E^{H}F=E$.\\
The Galois closure of $E^{H}$ over $K$ is also contained in $\C_{p^\ell}$ (being $\C_{p^\ell}/K$ Galois) and it is the composite of $E^{H}$ with a suitable subextension of $F/K$, which is then contained in $\C_{p^\ell}$. Since as $E$ varies through the elements of $\A$ with $E/K$ Galois these subextensions generate $F$, we have $F\subseteq \C_{p^\ell}$ and hence $E\subseteq \C_{p^\ell}$.\\
Finally observe that these extensions $E$ are sufficient to generate $\A$, thus we have $\A\subseteq \C_{p^\ell}$.
\end{proof}
Lemma \ref{C=A} implies that $\C_{p^\ell}/F$ is a subextension of $\mathscr{A}_{F}/F$, where $\mathscr{A}_{F}$ is the maximal abelian extension of $F$ of exponent $p$. Thus $[\C_{p^\ell}:F]=p^t$ for some $t\leq n_{F}+2$ (see \cite{DD}). By Kummer theory we know that $\C_{p^\ell}$ corresponds to a subgroup of $F^*/{F^*}^p$ and $[\C_{p^\ell}:F]$ is equal to the order of this subgroup (clearly $F$ contains the $p$-th roots of $1$).
\begin{prop}
One has $$[\C_{p^\ell}:F]=
\begin{cases}
p^{((p^\ell-1)^2-(p-1)^2)n_K} & \text{if}\ \ell\mid f_K\\
p^{(\ell+1)(p^\ell-p)(p-1)n_K} & \text{if}\ \ell\nmid f_K\\
\end{cases}.$$ 
\end{prop}
\begin{proof}
By Lemma \ref{C=A}, $[\C_{p^\ell}:F]=[\A:F]$ and $\A=F(\sqrt[p]{\Delta})$, where $\Delta$ is the $\F_p[H]$-submodule of $F^*/{F^*}^p$ generated by the irreducible submodules of dimension $\ell$. It follows that $[\C_{p^\ell}:F]$ is the order of $\Delta$.\\
$\Delta$ can be generated by a finite number of disjoint irreducible $\F_p[H]$-submodules of dimension $\ell$, clearly if $m$ is this number then $[\C_{p^\ell}:F]=p^{\ell m}$.\\
We know that $$F^*/{F^*}^p\simeq \F_p\oplus\bigoplus_{i\in [[0,I_F]]}M_i\oplus M_\omega $$
as $\F_p[H]$-modules; the $\F_p[H]$-modules of dimension $\ell$ are contained in $\bigoplus_{i\in [[0,I_F]]}M_i$. Moreover we know that a similar module $V$ is the sum of the conjugates of a certain $J_{(\alpha,\beta)}$ which is an $\overline\F_p[H]$-submodule of $\overline{M}_i=M_i\oplus_{\F_p}\overline\F_p$; as $\alpha$ and $\beta$ vary in $\F_{p^\ell}^*$ we find all the irreducible representation of dimension $\ell$ contained in $F^*/{F^*}^p$. The multiplicity of $V$ in $\bigoplus_{i\in [[0,I_F]]}M_i$ is equal to that of $J_{(\alpha,\beta)}$ in $Y=\bigoplus_{i\in [[0,I_F]]}\overline{M}_i$, it counts the number of disjoint $\F_p$-vector space (resp. $\overline \F_p$-vector space) in $F/{F^*}^p$ on which $H$ acts as on $V$ (resp. $J_{(\alpha,\beta)}$).\\
In particular, if $\ell\mid f_K$ then to achieve $\F_p$-representation of dimension $\ell$ we must have $(\alpha,\beta)\in (\F_{p^\ell}^*\times\F_{p^\ell}^*)\setminus (\F_p^*\times\F_p^*)$ and for each of these pairs $J_{(\alpha,\beta)}$ has dimension $1$ and multiplicity $n_K$. As usual observe that the $\F_p$-representation containing $J_{(\alpha,\beta)}$ is equal to that containing its other $\ell-1$ conjugates that is $J_{(\alpha^p,\beta^p)},\dots,J_{(\alpha^{p^{\ell-1}},\beta^{p^{\ell-1}})}$. Therefore we have $$m=\frac{1}{\ell}[(p^\ell-1)^2-(p-1)^2]n_K.$$
If $\ell\nmid f_K$ then only one between $\alpha$ and $\beta$ is in $\F_{p^\ell}^*\setminus \F_p^*$ and the other in $\F_p^*$. The $J_{(\alpha,\beta)}$'s with $\alpha\in\F_{p^\ell}^*$ and $\beta\in\F_p^*$ have no conjugates different from itself and multiplicity $\ell n_K$, but $J_{(\alpha,\beta)}=J_{(\alpha^p,\beta)}=\dots =J_{(\alpha^{p^{\ell-1}},\beta)}$. While the $J_{(\alpha,\beta)}$'s with $\alpha\in\F_p^*$ and $\beta\in\F_{p^\ell}^*$ have multiplicity $n_K$ and $\ell$ conjugates like the previous case. Thus we have
$$m=(p^\ell-p)(p-1)n_K+\frac{1}{\ell}(p-1)(p^\ell-p)n_K.$$
Substituting in $[\C_{p^\ell}:F]=p^{\ell m}$ we obtain the thesis.
\end{proof}
Observe that by the Theorem \ref{mainimpr} $\G\simeq G\rtimes H$, where $G$ is the Galois group of $\C_{p^\ell}$ over $F$.
It is well known that $\G_{-1}=\G$, $\G_{0}$ is the inertia subgroup of $\G$ and $\G_1$ is the $p$-Sylow subgroup of $\G_0$. Since $\C_{p^\ell}/F$ has no unramified subextensions being the composite of totally ramified extensions, $\G_0$ is isomorphic to $G\rtimes H_0$ where $H_0$ is the inertia subgroup of $H$, while $\G_i\simeq G_i$ for all $i\geq 1$ since $F/K$ is tame.\\
It is suitable to search the upper numbering ramification groups because, if $L$ is a Galois extension of $F$ contained in $\C_{p^\ell}$ and $G_L=\mathrm{Gal}(\C_{p^\ell}/L)$ then for all $v$ one has $(G/G_L)^v\simeq G^v G_L/G_L$ (see \cite{Ser}, Ch. IV, \S 3, Prop. 14) and, in our case, if also $L$ has degree $p^\ell$ over $F$ and is Galois over $K$ the groups $(G/G_L)^v$ are easy to describe.
\begin{lemma}\label{ram on F}
Let
\begin{equation*}
d=
\begin{cases}
((p^\ell-1)^2-(p-1)^2)n_K & \text{if $\ell\mid f_K$}\\
(\ell+1)(p^\ell-p)(p-1)n_K & \text{if $\ell\nmid f_K$}
\end{cases}
\end{equation*}
One has
\begin{equation*}
\mathrm{dim}_{\F_p}G^v=
\begin{cases}
d & \text{if $-1\leq v\leq 1$}\\
d-\left(\lceil v\rceil-\left\lceil\frac{v}{p}\right\rceil\right)f_F & \text{if $1<v\leq \frac{pe_F}{p-1}-1$}\\
0 & \text{if $v>\frac{pe_F}{p-1}-1$}
\end{cases}
\end{equation*}
\end{lemma}
\begin{proof}
First of all observe that $\C_{p^\ell}/F$ is a totally and wildly ramified extension, therefore $G_{-1}=G_0=G_1=G$; moreover a simple calculation of the Herbrand's function (see \cite{Ser}, Ch. IV, \S 3) yields to $G^1=G_1$ so that for $-1\leq v\leq 1$ one has $G^v=G=\mathrm{Gal}(\C_{p^\ell}/F)$ and $\mathrm{dim}_{\F_p}G^v=d$. Moreover, since $\C_{p^\ell}/F$ is an elementary abelian $p$-extension, by Theorem 12 of \cite{CDC}, for every $v>{pe_F}/{p-1}$ we have $G^v=1$, i.e. $\mathrm{dim}_{\F_p}G^v=0$. It remains to determine $\mathrm{dim}_{\F_p}G^v$ when $1<v\leq {pe_F}/{p-1}$.\\
Let $L$ be an abelian $p^\ell$-extension of $F$ of exponent $p$ having no subextensions that are Galois over $K$ and let $G_L=\mathrm{Gal}(\C_{p^\ell}/L)$. \\
As remarked before, we know that for all $v$ one has $(G/G_L)^v\simeq G^v G_L/G_L$, hence
$$G^v=\bigcap_{(G/G_L)^v=1}G_L.$$
Via Galois theory, this group corresponds to $\C_{p^\ell}^{G^v}=\prod L$, with the composite taken over the abelian $p^\ell$-extensions $L/F$ satisfying the previous condition, i.e. such that $(G/G_L)^v=1$. It follows that $\mathrm{dim}_{\F_p}G^v$ is equal to $d$ minus the dimension of the $\F_p$-vector space associated to $\prod L$ by Kummer theory.\\
Since $G/G_L\simeq \mathrm{Gal}(L/F)\simeq (\Z/p\Z)^\ell$, by Proposition 7 in \cite{CDC} $L/F$ has an upper ramification jump in $t$ if and only if there exists a subextension $N/F$ of degree $p$ with upper ramification jump equal to $t$; moreover the jumps can only appear in the integers $t$ such that $1\leq t\leq {pe_F}/{p-1}$ and $(t,p)=1$. Therefore, for all $v>{pe_F}/{p-1}$, $(G/G_L)^v=1$ for every abelian $p^\ell$-extension $L/F$.\\
Let $v$ be a real number such that $1<v\leq {pe_F}/{p-1}$, then $G^v=G^{\lceil v\rceil}$ and $G^{\lceil v\rceil}$ is the subgroup of $G$ fixing the composite of all the abelian $p^\ell$-extensions $L/F$ which have a jump in $\mathcal{J}=\{1,\dots,\lceil v\rceil -1\}$. Using Proposition 14 of \cite{CDC} and the analysis of the 
possible representation associated to $L$ we have made in section \ref{reprirr}, we find that there are $f_F$ disjoint extensions $L/F$ which has a jump in a fixed $t\in\mathcal{J}$ and $(t,p)=1$. Since the possible $t$ are $\lceil v\rceil-\left\lceil\frac{v}{p}\right\rceil$, we have the thesis.
\end{proof}
Note that our results is in agreement with Lemma 9 of \cite{DD}.\\
We can now prove the following
\begin{prop}
Let $d$ be as in previous Lemma \ref{ram on F}.
The ramification groups $\G_i$ of $\C_{p^\ell}/K$ are:
\begin{align*}
\G_{-1} & = \G = G\rtimes H  &{}\\
\G_0   & = G\rtimes H_0 &{}\\
\G_i  & =(\Z/p\Z)^{d-kf_F} & \text{if\ }  & t(k-1)<i\leq t(k)\\
\G_i & = \{e\}   & \text{if\ } & i>t(e_F-1)
\end{align*}
where $t(-1)=0$, $t(0)=1$ and for every $1\leq k\leq e_F-1$,
\begin{equation*}
t(k)=
\begin{cases}
t(k-1)+p^{kf_F} & \text{if $k\not\equiv 0\ (\mathrm{mod}\ p-1)$}\\
t(k-1)+2p^{kf_F} & \text{if $k\equiv 0\ (\mathrm{mod}\ p-1)$}\\
\end{cases}
\end{equation*}
Therefore there are $e_F+2$ jumps in the lower ramification groups.
\end{prop}
\begin{proof}
The claim for $\G_{-1}$ and $\G_0$ is clear by the considerations we have made before, $\G_1=G_1$ since $\C_{p^\ell}/F$ is wildly ramified and so $F/K$ is the maximal tame subextension. The rest of the Proposition follows by Lemma \ref{ram on F} with simple calculations, recalling that $\G_i=G_i$ for all $i\geq 1$ and the Herbrand's function (see \cite{Ser}, Ch. IV, \S 3)
$$G^v=G_{\psi (v)}\quad \text{with}\quad \psi(v)=\int_{0}^{v}(G^0:G^w)dw$$
which relates the upper and the lower ramification filtrations.\\
Lemma \ref{ram on F} says that $G$ has jumps in the upper ramification in every integers of $[[0,\frac{pe_F}{p-1}]]$, i.e. in the integers of the set $\{1\leq v\leq \frac{pe_F}{p-1}-1\mid v\not\equiv 0\ (\mathrm{mod}\ p)\}$; every such jump of $G$ gives a jump in the lower ramification of $\G$, and these together with $-1$ and $0$ are exactly the $e_F+2$ jumps of $\G$.
\end{proof}
Finally we determine the discriminant $\mathfrak{d}_{\C_{p^\ell}/K}$ of $\C_{p^\ell}/K$.
\begin{prop}
If $\pi_K$ is a uniformizer of $K$ and $d=[\C_{p^\ell}:F]$ as in the previous Lemma \ref{ram on F} then $$\mathfrak{d}_{\C_{p^\ell}/K}=(\pi _K)^\alpha$$
where
$$\alpha=f_{F/K}\left(\Bigl([F:K]-1+\frac{p(e_F+1)-1}{p-1}\Bigr)p^d-1-\frac{p^{n_F}-1}{p^{f_F}-1}-\frac{p^{n_F}-1}{p^{(p-1)f_F}-1}\right).$$
\end{prop}
\begin{proof}
Using the well known formula
\begin{equation*}
v_{\C_{p^\ell}}(\mathfrak{D}_{\C_{p^\ell}/K})=\sum_{i=0}^{\infty}(|\G_i|-1)
\end{equation*}
where $\mathfrak{D}_{\C_{p^\ell}/K}$ is the different of $\C_{p^\ell}/K$, one gets
\begin{align*}
v_{\C_{p^\ell}}(\mathfrak{D}_{\C_{p^\ell}/K}) & =[F:K]p^d-1+\sum_{k=0}^{e_F-1}{(p^{d-kf_F}-1)p^{kf_F}}+\\
& +\sum_{\substack{k=1\\ k\equiv 0\ (\mathrm{mod}\ p-1)}}^{e_F-1}{(p^{d-kf_F}-1)p^{kf_F}}.
\end{align*}
The conclusion follows by a simple calculation, recalling that $\mathfrak{d}_{\C_{p^\ell}/K}=N_{\C_{p^\ell}/K}(\mathfrak{D}_{\C_{p^\ell}/K})$ where $N_{\C_{p^\ell}/K}$ is the norm map and observing that, since $\C_{p^\ell}/F$ is totally ramified, $N_{\C_{p^\ell}/K}(\pi_{\C_{p^\ell}})=(\pi_K)^{f_{F/K}}$ ($\pi_{\C_{p^\ell}}$ a uniformizer of $\C_{p^\ell}$).
\end{proof}
\bibliographystyle{plain}
\nocite{*}
\bibliography{biblios}
\end{document}